\documentclass[preprint]{elsarticle}
\usepackage{amsmath,amssymb,amscd,amsthm}
\let\phi\varphi
\let\kappa\varkappa
\let\le\leqslant
\let\ge\geqslant

\let\emptyset\varnothing

\let\ES\varnothing

\newcommand\Int{\mathop{\mathrm{Int}}}
\newcommand\Cl{\mathop{\mathrm{Cl}}}

\newcommand{\sub}{\mathop{\mathrm{sub}}}

\newcommand{\pr}{\mathop{\mathrm{pr}}\nolimits}

\newcommand{\diam}{\mathop{\mathrm{diam}}}

\newcommand{\uni}[1]{\mathbf{1}_{{#1}}}

\newcommand\subcl{\mathrel{\underset{\mathrm{cl}}{\subset}}}
\newcommand\subop{\mathrel{\underset{\mathrm{op}}{\subset}}}

\newcommand\ups{{\uparrow}}
\newcommand\dns{{\downarrow}}

\newcommand{\Comp}{\mathcal{C}\mathrm{omp}}

\newcommand{\BBN}{\mathbb{N}}

\newcommand{\BBR}{\mathbb{R}}

\newcommand{\CCA}{\mathcal{A}}
\newcommand{\CCB}{\mathcal{B}}
\newcommand{\CCC}{\mathcal{C}}

\newcommand{\CCF}{\mathcal{F}}

\newcommand{\CCR}{\mathcal{R}}

\newcommand{\CCV}{\mathcal{V}}

\newcommand\Camb{\mathcal{CA}\mathrm{mb}}
\newcommand\Csamb{\mathcal{CSA}\mathrm{mb}}
\newcommand\Cpamb{\mathcal{CPA}\mathrm{mb}}
\newcommand\Cpsamb{\mathcal{CPSA}\mathrm{mb}}
\newcommand\Coamb{\mathcal{COA}\mathrm{mb}}
\let\sms\smallsmile
\let\ccirc\circledcirc
\def\adot{\mathbin{\underset{\raisebox{.25ex}{\rm *}}{\odot}}}
\def\acirc{\mathbin{\underset{\raisebox{.2ex}{\rm *}}{\circledcirc}}}
\def\bcirc{\mathbin{\bar\circledcirc}}
\def\bacirc{\mathbin{\underset{\raisebox{.2ex}{\rm *}}{\bar\circledcirc}}}
\def\ssms{{%
\rlap{$\scriptstyle\smallsmile$}%
\raise .5ex \hbox{$\scriptstyle\smallsmile$}%
}}

\newtheorem{theo}{Theorem}[section]
\newtheorem*{theo*}{Theorem}
\newtheorem{lem}[theo]{Lemma}
\newtheorem{stat}[theo]{Proposition}
\newtheorem{cons}[theo]{Corollary}
\theoremstyle{definition}
\newtheorem{defn}[theo]{Definition}
\newtheorem*{defn*}{Definition}

\newtheorem*{rem*}{Remark}
\newtheorem{exam}[theo]{Example}
\newtheorem*{exam*}{Example}
\newtheorem{que}[theo]{Question}
\newtheorem*{que*}{Question}

\journal{Fuzzy Sets and Systems}

\begin{document}

\begin{frontmatter}




\title{Ambiguous representations as fuzzy relations between sets}

\tnotetext[supp]{This research was supported by the Slovenian Research
Agency grants P1-0292-0101, J1-2057-0101 and BI-UA/09-10-002, and
the Ministry of Science and Education of Ukraine project
M/95-2009.}


\author[nyk]{Oleh Nykyforchyn\corref{cor1}}
\ead{oleh.nyk@gmail.com}
\author[rep]{Du\v san Repov\v s}
\ead{dusan.repovs@guest.arnes.si}
\cortext[cor1]{}

\address[nyk]{Vasyl' Stefanyk Precarpathian National University,
Shevchenka 57, Ivano-Frankivsk, 76025, Ukraine}
\address[rep]{Faculty of Mathematics and Physics and Faculty of Education,
University of Ljubljana, Jadranska 19, Ljubljana, 1000, Slovenia}

\begin{abstract}
Crisp and $L$-fuzzy ambiguous representations of closed subsets of
one space by closed subsets of another space are introduced. It is
shown that, for each~pair of compact Hausdorff spaces, the set of
(crisp or $L$-fuzzy) ambiguous representations is a lattice and a
compact Hausdorff Lawson upper semilattice. The categories of
ambiguous and $L$-ambiguous representations are defined and
investigated.
\end{abstract}


\end{frontmatter}


\section*{Introduction}

The necessity of modeling various kinds of uncertainty, imprecision
and incompleteness of information has resulted in a variety of
theories which in most cases either can be reduced to two main
ideas --- fuzziness and roughness, or they combine the two in
different ways.

A set is said to be fuzzy if, for an arbitrary element, its
membership can be not only completely true or completely false, but
also intermediate membership grades can occur. This level of
membership can be expressed as a number in the range $[0,1]$
(classical fuzzy sets \cite{Zadeh:FuzzySets:65}), as a subinterval
of $[0,1]$ (interval based fuzzy sets
\cite{GehWalWal:CommIntBasFuzSets:96}), as a pair of two numbers
with the sum $\le 1$ that indicate our confidence in its membership
and non-membership (vague sets, \cite{GauBue:VagueSets:93}), as a
mapping from $[0,1]$ to $[0,1]$ (type 2 fuzzy sets,
\cite{Zadeh:ConcLingVar:75}), as an element of a lattice ($L$-fuzzy
sets, \cite{Goguen:LFuzzySets:67}), etc. A~membership grade of $x$
in $F$ can be interpreted in different
ways~\cite{DubPr:ThreeSemFuzzSets:97}, e.g.\ as proximity of $x$ to
``prototype elements'' of $F$, as plausibility, certainty or truth
degree of ``$x$ is in $F$'', as ease (``cost'') of making $x$ to
``fit'' into $F$ etc. We shall neither discuss nor compare
different semantics of fuzzy sets, but accept a~convention, which
is compatible with all of them: ``the~more'' a~membership degree
is, ``the~better'' a~respective element fits into a~considered
class of objects. Hence in the~sequel ``truth value'' of a~sentence
can be understood not necessarily in the~strict sense of
multivalued logic, but also as degree of acceptability or
certainty etc.

A set is said to be rough if it is contained in a universe where one
sometimes cannot distinguish one element from another, which is
formalized via partitions or equivalence relations
\cite{Pawl:RoughSets:82}. In a rough set some elements are definitely
contained, some are decidedly not members, whereas for
some elements the answer is uncertain.

Nevertheless, these concepts do not completely cover uncertainty and
imprecision in the~description of \emph{sets}. The existing tools
consider it to be a derivative of uncertainty and imprecision in
description\-/recognition\-/membership of individual elements
\cite{Li:RouApprOper:08,Lin:NeighSys:97}, which is insufficient.
Let us imagine taking a digital photo of a text when the camera is
subject to random small shifts. Then the image of a character is
not uniquely determined and in order to recognize it, we cannot
compare a region of a photo with a pattern on the per-pixel basis.
Moreover, if all pixels of the candidate image are obtained from
the pattern by a shift of 1 pixel to the left, the result is much
more acceptable than if all pixels of the pattern are moved by 1
pixel in random directions, although the two possibilities are
equivalent from the "elementwise" point of view.

Even though fuzzy mathematics deals with sets, e.g. when fuzzy
variants of subsetness, similarity and distance between sets are
investigated \cite{KehKon:LFuzzyIncSimDist:03}, mostly (with rare
exceptions~\cite{Li:RouApprOper:08}) relations between subsets of
\emph{the same} universe are considered. We believe that a
fundamental distinction between an object and its observable image,
e.g. between a 3-D body and its 2-D photo, has to be reflected in
an adequate theory.

In is also important that fuzzy/rough theoretical investigations
consider continuity and other topological properties of the
procedures suggested to process uncertain and imprecise data
(probably with only finite sets in mind). Ignoring this is rather
risky because small inaccuracy can cause incorrect conclusions.

In this paper we propose a notion of a~\emph{(fuzzy) ambiguous
representation} of subsets of one universe by subsets of another
one. Sets under consideration are
closed subsets of compact Hausdorff spaces, which in most cases is
sufficient for applications, e.g.\ all closed bounded sets of
$\BBR^n$ fit into this case. Recall also that each finite set
can be regarded as a compact Hausdorff space with the discrete
topology. Hence, if the reader wants to quickly gain an idea
about the~introduced objects, he/she can apply the following to finite
sets only (and to \emph{all} their subsets) and skip all
topological issues. On the~other hand, although it is desirable to
extend our results e.g.\ to Tychonoff spaces or complete
metric spaces (and we will do this in the future), this
extension raises many complications, cf.~\cite{NykRep:IncHypCapTych:09}
on a~similar problem for inclusion hyperspaces and capacities,
which are "building blocks" for ambiguous representations.

All compact Hausdorff
spaces and (fuzzy) ambiguous representations are arranged into
categories, thus allowing one to compose representations and (in
some cases) to find representations that are inverse to a given one
(in a special sense). It is shown that the set of "good" (fuzzy)
ambiguous representations between fixed compact Hausdorff spaces is
a lattice and a compact Hausdorff space as well.

The paper is organized as follows. First, all necessary definitions
and facts (or references to sources) are provided in Section~1.
This is followed by a strict mathematical exposition of (crisp)
ambiguous representations in Section~2 and $L$-fuzzy ambiguous
representations in Section~3. In Section~4 we discuss
possible interpretations and applications.

\section{Preliminaries}

In the sequel a \emph{binary (ternary) relation} means an arbitrary
subset of the product of two (resp.\ three) sets. If these sets are
topological spaces, we call a relation \emph{closed} if it is a
closed set in the product topology. For a binary relation $R\subset
X\times Y$ and elements $a\in X$ and $b\in Y$ we denote $aR=\{y\in
Y\mid (a,y)\in R\}$, $Rb=\{x\in X\mid (x,b)\in R\}$.

For relations $R\subset X\times Y$ and $S\subset Y\times Z$, the
composition of $R$ and $S$ is defined in the usual way, i.e.\ as
$$
\{(x,z)\in X\times Z\mid\text{there is }y\in Y
\text{ such that }(x,y)\in R,(y,z)\in S\}.
$$
The~obtained relation is often denoted by $R\circ S$ (cf.\ e.g.\
\cite{Winter:CatApprToLFuzzyRel:07}), but this contradicts to
the~notation for compositions of mappings. If $R\subset X\times Y$
is such that, for each $x\in X$, there is a~unique $y\in Y$ such
that $(x,y)\in R$, then $R$ is a~mapping $X\to Y$, and
the~mentioned $y$ is regarded as the~value $R(x)$. If $S\subset
Y\times Z$ is also a~mapping $Y\to Z$, then the~composition of
$R:X\to Y$ and $S:Y\to Z$ is a~mapping $X\to Z$, which is usually
denoted by $S\circ R$. To avoid confusion, we denote
the~composition of relations $R\subset X\times Y$ and $S\subset
Y\times Z$ by $R\ccirc S$ (or by other similar symbols with extra
circles), hence $R\ccirc S=S\circ R$ for mappings $R,S$.

Let $L$ be a~complete distributive lattice with a~bottom element
$0$ and a~top element $1$. An~$L$-\emph{fuzzy set} $F$ in
a~universe $X$ is a~mapping $F:X\to L$, with $F(x)$ being
interpreted as the truth degree of the~fact $x\in F$. Similarly,
an~$L$-\emph{fuzzy binary relation} $R$ between elements of
universes $X$ and $Y$ is a~mapping $R:X\times Y\to L$,
cf.~\cite{Winter:CatApprToLFuzzyRel:07}; we use either $xRy$ or
$R(x,y)$ to denote the truth degree of the sentence ``$x$ and $y$
are related by $R$''. If $F(x)$ (or $R(x,y)$) takes only values $0$
and $1$, then the~respective set (or the~relation) is called
\emph{crisp} and is identified with the~(usual) set $\{x\in X\mid A(x)=1\}$
(resp. with the~binary relation $\{(x,y)\in X\times Y\mid
R(x,y)=1\}$).

For an~$L$-fuzzy set $F:X\to L$ and $\alpha\in L$, the~(strong)
$\alpha$-\emph{cut}~\cite{TeKiMi:FuzzySysAndApps:92} of $F$ is
the~set $F_\alpha=\{x\in X\mid F(x)\ge \alpha\}$, which is
identified with the~crisp set
$$
F_\alpha(x)=
\begin{cases}
1,F(x)\ge \alpha,\\0,F(x)\not\ge\alpha,
\end{cases}
\quad
x\in X.
$$
Similarly, for an~$L$-fuzzy binary relation $R:X\times Y\to L$ and
$\alpha\in L$, the~$\alpha$-\emph{cut} of
$R$~\cite{Winter:CatApprToLFuzzyRel:07} is defined as the~crisp
relation:
$$
R_\alpha(x,y)=
\begin{cases}
1,R(x,y)\ge \alpha,\\0,R(x,y)\not\ge\alpha,
\end{cases}
\quad
x\in X,y\in Y.
$$
By the~above, from now on we identify $R_\alpha$ with the~binary
relation $\{(x,y)\in X\times Y\mid R(x,y)\ge\alpha\}$.

It is obvious that the~family $(F_\alpha)_{\alpha\in L}$ of
$\alpha$-cuts (so called $L$-\emph{flou
set}~\cite{NegRal:RepTheo:75}) determines an~$L$-fuzzy set $F:X\to
L$ completely, as well as the~family $(R_\alpha)_{\alpha\in L}$
completely determines an~$L$-relation $R:X\times Y\to L$. It is
natural to ``collect'' these families into single subsets of
$X\times L$ and $X\times Y\times L$, respectively. Hence we
identify each $L$-fuzzy subset of $X$ and each $L$-relation between
elements of $X$ and $Y$ with their
\emph{subgraphs} (or \emph{hypographs})
$$
\sub F=\{(x,\alpha)\in X\times L\mid \alpha\le F(x)\}
$$
and
$$
\sub R=\{(x,y,\alpha)\in X\times Y\times L\mid \alpha\le R(x,y)\}
$$
respectively.

A~triple $(x,y,\alpha)$ is in $\sub R$ if and only if the~truth
degree of the sentence ``$x,y$ are related by $R$'' is at
least~$\alpha$. For a~set $S\subset X\times Y\times L$ to be
a~subgraph of an~$L$-relation, necessary and sufficient conditions
are:

(1) $S\supset X\times Y\times \{0\}$;

(2) for $x\in X$, $y\in Y$, and $A\subset L$ such that
$(x,y,\alpha)\in S$ for all $\alpha\in A$, the~triple $(x,y,\sup
A)$ is also in $S$.

An~obvious similar condition is valid also for subgraphs of
$L$-fuzzy sets.

To define compositions of $L$-fuzzy relations, we
follow~\cite[Section~3.3]{Winter:CatApprToLFuzzyRel:07} and require
that $L$ is a~complete lattice, an~operation $*:L\times L\to L$ is
associative, commutative, infinitely distributive w.r.t.\
``$\lor$'' in the both arguments, and $1$ is a neutral element for
"$*$" (i.e.\ $(L,*,0,1)$ is a~\emph{commutative lattice-ordered
semigroup} in the~terminology of
\cite{Winter:CatApprToLFuzzyRel:07}). Then, for fuzzy relations
$R:X\times Y\to L$ and $S:Y\times Z\to L$, the~\emph{composition}
$R\acirc S:X\times Z\to L$ is defined by the~formula
$$
R\acirc Q(x,z)=\sup\{R(x,y) *R(y,z)\mid y\in Y\},\; x\in X,z\in Z.
$$
Therefore
\begin{multline*}
\sub(R\acirc Q)=
\bigl\{(x,z,\alpha)\in X\times Z\times L
\mid
\alpha\le\sup
\{\beta *\gamma\mid
\\
\text{there is }y\in Y
\text{ such that }(x,y,\beta)\in \sub R,(y,z,\gamma)\in\sub Q
\}
\bigr\}.
\end{multline*}

>From now on we shall often treat $L$-fuzzy sets and $L$-relations
as subgraphs and write $F$ and $R$ instead of $\sub F$ and $\sub
R$. The~latter formula will be considered as the~definition of
composition.

We write $A\subop X$ (or $A\subcl X$) if $A$ is an open (resp.
closed) set in a topological space $X$. A~\emph{compactum} is a
(not necessarily metrizable) compact Hausdorff space.
The~\emph{category of compacta} $\Comp$ consists of all compacta
and continuous mappings between them (cf. \cite{ML:CWM:98} for
definitions of a category and a functor). For a compactum $X$, its
\emph{hyperspace} $\exp X$ consists of all non-empty closed subsets
of~$X$. We will use the (de-facto) standard
\emph{Vietoris topology}~\cite{Mich:TopOnSpSubsets:51} on $\exp X$
with a base that consists of all the sets of the form
$$
\langle U_1,\dots,U_n\rangle=
\{F\in\exp X\mid F\subset U_1\cup\dots\cup U_n,
F\cap U_i\ne\ES, i=1,\dots,n\},
$$
with $n\in\BBN$, $U_1,\dots,U_n\subop X$. The space $\exp X$ is a
compactum as well, hence we can write $\exp^2X=\exp(\exp X)$ etc.
The Vietoris topology is the least upper bound of the~\emph{upper
topology} with a subbase $\{\langle U\rangle\mid U\subop X\}$, and
the \emph{lower topology} with a subbase $\{\langle X,U\rangle\mid
U\subop X\}$. A continuous mapping into $\exp X$ with the upper
(lower) topology is called
\emph{upper} (resp.\ \emph{lower}) \emph{semicontinuous}.

If a mapping $f:X\to Y$ of compacta is continuous, then the mapping
$\exp f:\exp X\to\exp Y$, $\exp f(F)=\{f(x)\mid x\in F\}$ for all
$F\in\exp X$, is well defined and continuous. Thus the
\emph{hyperspace functor} $\exp$ in $\Comp$ is obtained.

A closed non-empty subset $\CCA\subset\exp X$ is called
an~\emph{inclusion hyperspace}~\cite{Mo:InclHyp:88} if for $A,B\in
\exp X$, the inclusion $B\supset A\in\CCA$ implies $B\in\CCA$. The
set $GX$ of all inclusion hyperspaces is closed in $\exp^2X$,
therefore with the induced topology it is a compactum. This topology
is determined with a subbase that consists of all the sets of the
form
$$
U^+=\{\CCA\in GX\mid \text{there is }A\in\CCA\text{ such that
}A\subset U\}
$$
and
$$
U^-=\{\CCA\in GX\mid A\cap U\ne\ES\text{ for all }A\in\CCA\},
$$
for all open $U\subset X$.

For any subset $\CCA\subset \exp X$ its
\emph{traversal} $\CCA^\perp=\{B\in\exp X\mid B\cap A\ne\ES
\text{ for all }A\in\CCA\}$ is an inclusion hyperspace, and the
correspondence $\CCA\mapsto \CCA^\perp$ is continuous and antitone
(with respect to inclusion). If $\CCA\subset\exp X$ contains all
closed supersets of its elements, then $(\CCA^\perp)^\perp =
\Cl\CCA$, hence $(\CCA^\perp)^\perp =\CCA$ if and only if $\CCA\in
GX$.

A topological upper (lower) semilattice is called
\emph{Lawson}~\cite{Laws:TopLatSmSl:69} if at each point it
possesses a local base consisting of upper (resp.\ lower)
subsemilattices. If $L$ is compact and Hausdorff, then this implies
that for each $F\subcl L$, $F\ne\ES$, the least upper (resp.
greatest lower) bound of $F$ exists and it continuously depends on
$F$. A \emph{Lawson lattice} is a distributive topological lattice
that is both an upper and a lower Lawson semilattice. In the sequel
all topological (semi)lattices will be considered Hausdorff. The
bottom and the top elements of a poset (if they exist) are denoted
by $0$ and $1$, respectively. By $\lor$ and $\land$ we denote resp.\
pairwise joins and meets. For a subset $A$ of a poset $L$, we
denote:
$$
A\dns=\{\beta\in L\mid\beta\le\alpha\text{ for some }\alpha\in A\},
\quad
A\ups=\{\beta\in L\mid\alpha\le\beta\text{ for some }\alpha\in A\}.
$$
For elements $\alpha,\beta$ of a poset $L$, we write
$\alpha\ll\beta$ and say that $\alpha$ is \emph{way below} $\beta$
if, for each directed set $D\subset L$ such that $\beta\le\sup D$,
there is an~element $\gamma\in D$ such that $\alpha\le\gamma$. If
$L$ is a~compact Lawson lattice, then this is equivalent to
$\beta\in\Int(\{\alpha\}\ups)$, hence to the~existence of
a~neighborhood $O_\beta\ni\beta$ such that $\alpha\le\inf O_\beta$.
The following statement is immediate:

\begin{lem}\label{lem.ldot}
Let $*:L\times L\to L$ be a continuous operation that is monotone
and satisfies infinite distributive laws w.r.t. $\inf$ in the both
arguments. For $\alpha,\beta,\gamma\in L$, if $\gamma\ll\alpha
*\beta$, then $\gamma\le\alpha'*\beta'$ for some
$\alpha',\beta'\in L$ such that $\alpha'\ll\alpha$,
$\beta'\ll\beta$.
\end{lem}

For a compact Lawson lattice $L$ and a compactum $X$, a function
$c:\exp X\cup\{\emptyset\}\to L$ is called an $L$-\emph{valued
capacity}~\cite{Nyk:CapLat:08} (or $L$-\emph{fuzzy measure}) on a
compactum $X$ if the following holds:
\begin{enumerate}
\item
$c(\emptyset)=0$, $c(X)=1$;
\item
for each closed subsets $F$, $G$ in $X$ the inclusion $F\subset G$
implies $c(F)\le c(G)$ (monotonicity); and
\item
if $F\subset X$ is closed and $c(F)$ lies in a neighborhood
$V\subset L$, then there exists an open subset $U\supset F$ such
that $c(G)\in V\dns$ for any closed $G\subset X$ satisfying
$G\subset U$ (upper semicontinuity).
\end{enumerate}

Denote by $M_LX$ the set of all $L$-valued capacities on a
compactum~$X$. We define a topology on $M_LX$ by a subbase that
consists of all sets of the form
\begin{multline*}
O_+(U,V)=\{c\in M_LX\mid
\text{ there is }F\subcl U\text{ such that }c(F)\ge
\alpha\text{ for some }\alpha\in
V\}=\\
=
\{c\in M_LX\mid
\text{ there is }F\subcl U,c(F)\in V\ups\},
\end{multline*}
where $U\subop X$, $V\subop L$, and
$$
O_-(F,V)=\{c\in M_LX\mid c(F)\le \alpha\text{ for some }\alpha\in
V\}=
\{c\in M_LX\mid
c(F)\in V\dns\},
$$
where $F\subcl X$, $V\subop L$.

It was proved in \cite{Nyk:CapLat:08} that the defined topology on
$M_LX$ is compact Hausdorff. If we take a subbase that consists
only of all elements of the first (second) form, we obtain the
\emph{upper} (resp. \emph{lower}) \emph{topology} on $M_LX$. Upper
(lower) semicontinuous functions into $M_LX$ are defined in the
obvious way.

The \emph{subgraph}~\cite{Nyk:CapLat:08} (or \emph{hypograph}) of a
capacity $c\in M_LX$ is a set $\sub c=\{(F,\alpha)\mid\ F\in\exp
X,\alpha\in L,\alpha\le c(F)\}\subset\exp X\times L$.

\begin{lem}\cite{Nyk:CapLat:08}
Let $X$ be a compactum, and $L$ a compact Hausdorff upper
semilattice that contains the greatest element. A subset
$\CCF\subset\exp X\times L$ is the subgraph of an $L$-valued
capacity if and only if for all closed nonempty subsets $F$, $G$ of
$X$ and all $\alpha,\beta\in L$ the following conditions are
satisfied:

1)
if $(F,\alpha)\in \CCF$, $\alpha\ge \beta$, then $(G,\beta)\in
\CCF$;

2)
if $(F,\alpha),(G,\beta)\in \CCF$, then $(F\cup
G,\alpha\lor\beta)\in S$;

3)
$\CCF\supset \exp X\times\{0\}\cup \{X\}\times L$; and

4)
$\CCF$ is closed.

\noindent
If these conditions hold, then the capacity $c$ is unique (and we
denote it $c_{\CCF}$).
\end{lem}

It was also proved in \cite{Nyk:CapLat:08} that the mapping
$\sub:M_LX\to\exp(\exp X\times L)$ is an embedding.


\section{Crisp ambiguous representations}

We need several technical results.

\begin{lem}\label{int-G}
Let $X$ be a compactum, and let a subset $G\subset \exp X$ be such
that $A\subset A'\subcl X$, $A\in G$ implies $A'\in G$. Then $G$ is
closed if and only if, for each filtered collection $\CCF$ of
elements of $G$, we have $\bigcap\CCF\in G$.
\end{lem}

\begin{proof} {\sl Necessity.}
We can regard the aforementioned $\CCF$ as a net that converges to
$\bigcap\CCF$, hence, for a closed $G$, the inclusion $\CCF\subset
G$ implies $\bigcap \CCF\in G$.

{\sl Sufficiency.} Let $A\in\exp X$ be a point of closure of $G$,
then for all $U\subop X$ such that $A\subset U$ there is $A'\in
\exp X$ such that $A'\in G$, and $A'\subset U$. This implies that
$\Cl U\in G$ for all open neighborhoods $U\supset A$. The closures
of these neighborhoods form a filtered collection with the
intersection $A$, hence, by assumption, $A\in G$.
\end{proof}

\begin{rem*}
A~\emph{non-empty} $G\subset \exp X$ that satisfies the conditions
of the previous lemma is precisely an inclusion hyperspace.
\end{rem*}

\begin{lem}\label{int-un-cl}
Let $X$ be a compactum, $L$ a compact Lawson upper semilattice and
let a subset $R\subset\exp X\times\exp Y$ be such that, for
$A,A'\in\exp X$, $\alpha,\alpha'\in L$, $A'\supset A$,
$\alpha'\le\alpha$, if $(A,\alpha)\in R$, then $(A',\alpha')\in R$.
Then $R$ is closed if and only if the following two conditions
hold:

1) for all $\alpha\in L$ and each filtered collection $\CCA$ of
elements of $\exp X$ such that $\CCA\times\{\alpha\}\subset R$, we
have $(\bigcap\CCA,\alpha)\in R$; and

2) for all $A\in\exp X$ the set of all $\alpha\in L$ such that
$(A,\alpha)\in R$ is closed.
\end{lem}

\begin{proof}
{\sl Necessity} of 1) is due to the previous lemma, and is obvious
for 2).

{\sl Sufficiency.} Suppose that 1), 2) hold, and let
$(A,\alpha)\in\exp X\times L$ be a point of the closure of $R$. For
any closed neighborhood $V\supset B$ and any open neighborhood
$O_\alpha\ni
\alpha$ in $L$, there are $A'\in\exp X$, $\alpha'\in L$ such that
$A'\subset V$, $\alpha'\in O_\alpha$, $(A',\alpha')\in R$, hence
$(V,\alpha')\in R$. Therefore $\alpha$ is a point of the closure of
the set of all $\alpha'\in L$ such that $(V,\alpha')\in R$, and, by
2), $(V,\alpha)\in R$. All closed neighborhoods $V\supset A$ form a
net that converges to $A$, thus by 1) we obtain $(A,\alpha)\in R$.
\end{proof}

Now we are ready to introduce the~main notion of this work.

\begin{defn}
An \emph{ambiguous representation} between compacta $X$ and $Y$ is
a subset $R\subset\exp X\times \exp Y$ such that:

a) if $A,A'\in\exp X$, $B,B'\in \exp Y$, $A'\subset A$, $B\subset
B'$, $(A,B)\in R$, then $(A',B')\in R$;

b) $(A,Y)\in R$ for all $A\in\exp X$; and

c) for all $A\in\exp X$ the set $AR=\{B\in\exp Y\mid (A,B)\in R\}$
is closed in $\exp Y$.
\end{defn}

Now we suggest a~simple model example.
\begin{exam}
Let $Y$ be the~set of students, say, of a~math department, and $X$
be the~set of all possible marks at a~test. Assume that, for
$A\subset X$ and $B\subset Y$, $A,B\ne\ES$, we have $(A,B)\in R$ if
it is likely that, after a~test, $A$ is a~\emph{subset} of
the~marks of the~students of $B$. Then $A$ in some sense represents
the~group $B$ of students, and it is obvious that a) holds.
The~property b) means that students' skills vary enough to result
in a~collection which contains \emph{any} possible marks. Since
both $X$ and $Y$ are finite, c) is immediate in this case, but in
general it means that, if $A$ can represent sets, which are
arbitrarily close to a~particular $B$, then $A$ is an~appropriate
representation for this~$B$.
\end{exam}

In the~last section there is a~more extensive discussion of
motivation for such a~definition. The~authors experienced
difficulties in choosing a~term for the~investigated relations, in
particular because most ``good words'' like ``fuzzy'', ``rough'',
``vague'' etc have been already ``occupied''. We needed to express
two features: one set is not necessarily a~copy or an~image of
another one (although this is also possible), but represents it in
some manner, which may vary; and there can be many valid
representatives for a~set, as well as a~set can represent many
sets, hence such a~relation is ambiguous. Thus we have arrived at
``ambiguous representations''.

Nonetheless, the~just introduced concept is not quite unrelated
to rough sets.
\begin{exam}
Let $\sim$ be an~\emph{indiscernibility relation} on a~finite set
$X$, i.e.\ it is an~equivalence relation that contains all pairs
of elements $x,y\in X$ such that we cannot distinguish $x$ from $y$.
Then two subsets $A,B$ are indistinguishable if and only if their
\emph{upper approximations}
$
\overline{appr}_{\sim}A=\{x\in X
\mid x\sim a\text{ for some }a\in A\}
$
and
$
\overline{appr}_{\sim}B=\{y\in X
\mid y\sim b\text{ for some }b\in B\}
$
are equal. Let $R\subset \exp X\times \exp X$ consist from all
pairs $(A,B)$ such that non-empty $A$ is indistinguishable from
a~subset of $B$, i.e. $\overline{appr}_{\sim}A\subset
\overline{appr}_{\sim}B$. Then $R$ is an~ambiguous representation
between $X$ and $X$.
\end{exam}

In the~next section the~concept of ambiguous representation will
also be modified in the~spirit of fuzzy mathematics.

We denote the~set of all ambiguous representations between $X$ and
$Y$ by $\Camb(X,Y)$.

\begin{defn}
If $R$ is an~ambiguous representation, and $(A,B)\in R$, we say
that $B$ is $R$-\emph{admissible} for $A$. If $C\subcl Y$, $C\cap
B\ne\ES$ for all $B\in AR$, we call $C$ an $R$-\emph{unavoidable}
set for $A$.
\end{defn}

\begin{exam}
Let $R\subset \exp X\times \exp X$ be the~relation defined in the~previous
example. Then $C\subset X$ is $R$-unavoidable for $A\subset X$ if
and only if the~upper approximation $\overline{appr}_{\sim}A$ has
non-empty intersection with the~\emph{lower approximation}
$
\underline{appr}{}_{\sim}C=\{c\in C
\mid x\in C\text{ for all }x\in X, x\sim c\}
$.
\end{exam}

\begin{exam}
Let $C$ be the~set of all students of a~department that are able to
obtain a~highest mark. If a~set $A$ contains such a~mark and
represents a~set $B$ of students, then $B$ must contain at least
one well-prepared student, hence $B\cap C\ne\ES$, and $C$ in
unavoidable for $A$.
\end{exam}

For an ambiguous representation $R\subset\exp X\times \exp Y$ we
define the relation $R^\sms\subset \exp Y\times\exp X$ as follows:
\begin{multline*}
R^\sms=\{(\tilde B,\tilde A)\in \exp Y\times\exp X\mid
\\
\forall A\in\exp X \;( A\cap \tilde A=\ES\implies
\exists B\in AR\;\; B\cap \tilde B =\ES)\},
\end{multline*}
i.e. $(\tilde B,\tilde A)\in R^\sms$ if and only if
$\tilde A$ has non-empty intersections with all $A\in\exp X$ such that
$\tilde B$ is $R$-unavoidable for $A$. Thus
$$
\tilde BR^\sms=\{A\in\exp X\mid \tilde B\in (AR)^\perp\}^\perp
$$
for all $\tilde B\in\exp Y$.

Observe that for $R^\sms$ conditions a),b) of the definition of an
ambiguous representation between $Y$ and $X$ obviously hold. The
previous formula implies c), hence $R^\sms\in\Camb(Y,X)$.

\begin{exam}
Let $X$ and $Y$ be the~squares $[0,1]\times [0,1]$ and
$[0,1]\times[1,2]$, respectively. For $A\in\exp X$, $B\in\exp Y$,
assume $(A,B)\in R$ if $pr_1(A)\subset \pr_1(B)$. In other words,
$A$ can represent $B$ if and only if it is covered by the~``shade''
of $B$ under vertical light. Then $(\bar B,\bar A)\in\exp Y\times
\exp X$ is in $R^\sms$ if and only if $\pr_1(X\setminus \bar
A)\subset\pr_1(Y\setminus \bar B)$, i.e.\ each point outside of
$\bar A$ is in the~``shade'' of a~point outside of $\bar B$.
\end{exam}

\begin{stat}
If $R\subset \exp X\times \exp Y$ is an ambiguous representation,
then $(R^\sms)^\sms\subset R$.
\end{stat}

\begin{proof}
For all $A\in\exp X$:
\begin{gather*}
A(R^\sms)^\sms=
\{\tilde B\in\exp Y\mid
A\in(\tilde BR^\sms)^\perp\}^\perp=
\\
\{\tilde B\in\exp Y\mid
A\in\{A'\in\exp X\mid \tilde B\in(A'R)^\perp\}^{\perp\perp}\}^\perp=
\\
\{\tilde B\in\exp Y\mid
A\in\Cl\{A'\in\exp X\mid \tilde B\in(A'R)^\perp\}\}^\perp=
\\
\{\tilde B\in\exp Y\mid
\tilde B\in(VR)^\perp\text{ for all }V\subcl X,A\subset\Int V\}^\perp\subset
\\
\{\tilde B\in\exp Y\mid
\tilde B\in(AR)^\perp\}^\perp=(AR)^{\perp\perp}=AR.
\end{gather*}
\vskip-18pt
\end{proof}

\begin{cons}\label{char-pseudo}
For an ambiguous representation $R\subset \exp X\times \exp Y$, the
equality $(R^\sms)^\sms=R$ is valid if and only if:

d) for all $(A,B)\in R$ and a closed neighborhood $V\supset B$
there is a closed neighborhood $U\supset A$ such that $(U,V)\in R$.
\end{cons}

\begin{proof}
>From the latter formula it is easy to obtain an explicit
expression for $(R^\sms)^\sms$: if $A\in\exp X$, then
$$
A(R^\sms)^\sms=\Cl(\bigcup \{VR\mid V\text{ is a closed neighborhood of
}A\}).
$$
\vskip-23pt
\end{proof}

\begin{defn}
If $R\subset \exp X\times \exp Y$ is an ambiguous representation
such that $(R^\sms)^\sms=R$, that we call $R^\sms$
\emph{pseudo-inverse} to $R$, and $R$ is called
\emph{pseudo-invertible}.
\end{defn}

\begin{rem*}
It is obvious that then $R$ is pseudo-inverse to $R^\sms$, and
$R^\sms$ is pseudo-invertible as well.
\end{rem*}

The set of all pseudo-invertible ambiguous representations between
$X$ and $Y$ is denoted by $\Cpamb(X,Y)$.

The composition of ambiguous representations $R\in\Camb(X,Y)$ and
$S\in\Camb(Y,Z)$ is defined in the usual way~:
\begin{gather*}
R\ccirc S=\{(A,C)\in\exp X\times \exp Z\mid \text{there is
}B\in\exp Y
\\
\text{ such that }(A,B)\in R,(B,C)\in S\}.
\end{gather*}

It is easy, however, to find an example of two ambiguous
representations (even pseudo-invertible ones) such that their
composition is not an ambiguous representation. To obtain a
category, we have two options: either to modify the composition
law, or to restrict the class of allowed relations.

For ambiguous representations $R\subset \exp X\times\exp Y$ and
$S\subset \exp Y\times\exp Z$, let a relation $R\bcirc S\subset\exp
X\times \exp Z$ be defined by the equality $A(R\bcirc S)=
\Cl(A(R\ccirc S))$ for all $A\in\exp X$ (the closure is taken w.r.t. the
Vietoris topology). In other words, $C\in A(R\bcirc S)$ if and only
if for all closed neighborhoods $V\subset C$ there is $B\in\exp Y$
such that $(A,B)\in R$, $(B,V)\in S$. Then $R\bcirc S$ is an
ambiguous representation. Unfortunately, the equality
$(R\bcirc S)\bcirc T=R\bcirc (S\bcirc T)$ is not valid in general,
hence compacta and ambiguous representations do not form a category.

\begin{lem}\label{lem.rs-sr}
For ambiguous representations $R\subset \exp X\times\exp Y$,
$S\subset\exp Y\times \exp Z$ the inclusion $S^\sms\bcirc
R^\sms\subset(R\bcirc S)^\sms$ is valid.
\end{lem}

The proof is straightforward.

\begin{stat}\label{cpamb-comp}
Let $R\subset\exp X\times \exp Y$ and $S\subset \exp Y\times\exp Z$
be pseudo-invertible ambiguous representations. Then $R\bcirc
S\subset\exp X\times\exp Z$ is a pseudo-invertible ambiguous
representation as well, and $(R\bcirc S)^\sms= S^\sms\bcirc
R^\sms$.
\end{stat}

\begin{proof}
By Lemma~\ref{lem.rs-sr},
$$
R\bcirc S=(R^\sms)^\sms\bcirc (S^\sms)^\sms \subset
(S^\sms\bcirc R^\sms)^\sms\subset (R\bcirc S)^\sms{}^\sms
\subset R\bcirc S,
$$
thus $(R\bcirc S)^{\sms\sms}=R\bcirc S$, and $(R\bcirc
S)^\sms=(S^\sms\bcirc R^\sms)^{\sms\sms}=S^\sms\bcirc R^\sms$.
\end{proof}

Due to Corollary~\ref{char-pseudo} the composition law $\bcirc$ is
associative for pseudo-invertible ambiguous representations. Hence
all \emph{compacta and pseudo-invertible ambiguous representations}
form a category $\Cpamb$. In this category a~pseudo-invertible
ambiguous representation $R\subset\exp X\times\exp Y$ is considered
as an~arrow $X\to Y$, and for an~arrow $S:Y\to Z$, i.e.\ for
a~pseudo-invertible ambiguous representation $S\subset\exp
Y\times\exp Z$, the~composition $S\circ R$ is equal to $R\bcirc S$.
For a compactum $X$, the identity arrow $\uni{X}$ is equal to
$\{(A,B)\in\exp X\times
\exp X\mid A\subset B\}$.

Observe also that $\uni{X}{}^\sms=\uni{X}$, therefore we obtain an
involutive antiisomorphism $(-)^\sms:\Cpamb^{op}\to\Cpamb$, which
preserves objects.

To use the second variant, i.e. to construct a class of relations
that is closed under composition, we add the following requirement:

e) for all $B\in\exp X$ the set $RB=\{A\in\exp X\mid (A,B)\in R\}$
is closed in $\exp X$.

\begin{defn}
A relation $R\subset\exp X\times \exp Y$ that satisfies a)--c) and
e) is called a~\emph{strict ambiguous representation} between $X$
and $Y$.
\end{defn}

Observe that $(\exp X,\subset)$ is a compact Lawson upper
semilattice, hence we can use Lemma~\ref{int-un-cl} to give an
equivalent definition:
\begin{defn}\label{def.sar-short}
A subset $R\subset\exp X\times \exp Y$ is a \emph{strict ambiguous
representation} if $R$ is closed, contains $\exp X\times \{Y\}$,
and $(A,B)\in R$, $A'\in\exp X$, $B'\in\exp Y$, $A'\subset A$,
$B'\supset B$ imply $(A',B')\in R$.
\end{defn}

We denote the set of all strict ambiguous representations between
$X$ and $Y$ by $\Csamb(X,Y)$. The intersection
$\Cpamb(X,Y)\cap\Csamb(X,Y)$ is denoted by $\Cpsamb(X,Y)$.

\begin{stat}
For compacta $X,Y$, a relation $R\subset \exp X\times \exp Y$ is an
ambiguous representation if and only if the correspondence
$A\mapsto AR$ is an antitone mapping from $\exp X$ to $GY$. This
mapping is:

$\bullet$ lower semicontinuous if and only if $R$ is
pseudo-invertible;

$\bullet$ upper semicontinuous if and only if $R$ is strict;

$\bullet$ continuous if and only if $R$ is pseudo-invertible and
strict.
\end{stat}

The proof is straightforward and it uses Lemma~\ref{int-G}.\qed

Since the composition of compact relations is compact, we
immediately obtain:

\begin{stat}
If $R\in\Csamb(X,Y)$ and $S\in\Csamb(Y,Z)$, then $R\ccirc
S\in\Csamb(X,Z)$.
\end{stat}

\begin{rem*}
For a reader to observe ambiguous representations "at work", we
also provide an independent proof of c) and d) (properties a), b)
are obviously satisfied).

Let us observe that $A(R\ccirc S)\subset \exp Z$ contains supersets
of all its elements and the entire space $Z$. To prove closedness,
we use Lemma~\ref{int-G}, assume that $\CCF\subset A(R\ccirc S)$ is
a filtered collection, and let $C_0=\bigcap \CCF$. Observe that
$(A,C)\in R\ccirc S$ if and only if $AR\cap SC\ne\ES$. By c)
$AR\subset \exp Y$ is an inclusion hyperspace, and by d) the
collection $\{SC\mid C\in\CCF\}$ is a decreasing net of closed
subsets of $\exp Y$ if $\CCF$ is ordered reverse to inclusion. Each
of $\{SC\mid C\in\CCF\}$ has a non-empty intersection with $AR$,
hence there is $B\in AR\cap\bigcap\{SC\mid C\in\CCF\}$. For any
neighborhood $U\subset C_0$ there is $C\in\CCF$ such that $C\subset
\Cl U$, therefore $(B,C)\in S$ implies $(B,\Cl U)\in S$. Using c)
again, we obtain $C_0\in A(R\ccirc S)$, i.e. c) for $R\ccirc S$.

Let $(A,B)\in\exp X\times\exp Z\setminus R\ccirc S$, then $AR\cap
SC=\ES$. Due to the compactness of $SC$ there are open sets
$U_1,\dots,U_n\subset Y$ such that all $B\in SC$ are contained in
some $U_i$, and no $\Cl U_i$ is an element of $AR$. Then
$A\in\CCV=\exp X\setminus(R\Cl U_1\cup\dots\cup R\Cl U_n)$, and
$\CCV$ is an open neighborhood of $A$ in $\exp X$ such that
$A'\in\CCV$ is incompatible with $(A',C)\in R\ccirc S$. Thus
$(R\ccirc S)C$ is closed. \qed
\end{rem*}

Therefore we can define the~\emph{category of compacta and strict
ambiguous representations} $\Csamb$. Arrows from $X$ to $Y$ in this
category are the~strict ambiguous representations between $X$ and
$Y$, the~composition of $R:X\to Y$ and $S:Y\to Z$ is the~usual
composition $R\ccirc S$ of relations, identity arrows are of the
same form as in $\Cpamb$. For strict ambiguous representations the
compositions $\ccirc$ and $\bcirc$ coincide, hence we can consider
the intersection of $\Cpamb$ and $\Csamb$, which we call
the~\emph{category of compacta and strict pseudo-invertible
ambiguous representations} and denote by $\Cpsamb$.

A drawback of introduction of e) is that the ambiguous
representation $R^\sms$ for $R\in \Csamb(X,Y)$ is not always
strict.

\begin{exam}
Let $f:X\to Y$ be a continuous mapping of compacta, and a relation
$R_f\subset \exp X\times \exp Y$ be defined as $(A,B)\in R_f$ if
and only if $f(A)\subset B$. Then $R_f$ is a pseudo-invertible
strict ambiguous representation, and for all $B\in\exp Y$ we have
\begin{gather*}
BR_f^\sms=
\{A\in\exp X\mid B\in \{f(A)\}^\perp\}^\perp=\\
\{A\in\exp X\mid f(A)\cap B\ne\ES\}^\perp=
\{A\in \exp X\mid A\supset f^{-1}(B)\}.
\end{gather*}

Therefore
$$
R_f^\sms A=\{\tilde B\in\exp Y\mid B\subset Y\setminus f(X\setminus
A)\}.
$$
It is easy to see that the latter set is closed for all closed $A$
if and only if the mapping $f$ is open.
\end{exam}

\begin{defn}
An ambiguous representation $R$ is called an \emph{open ambiguous
representation} if it is strict, pseudo-invertible and $R^\sms$ is
a strict ambiguous representation.
\end{defn}

\begin{rem*}
Here we do not mean that $R$ is an open subset in the product.
\end{rem*}

\begin{stat}
A relation $R\subset\exp X\times \exp Y$ that satisfies a)--e) is
an open ambiguous representation if and only if any of the
following statements is valid:

f) for any open $U\subset X$ the set of $R$-unavoidable sets of all
$A\subset U$, $A\in\exp X$, is open in the Vietoris topology on
$\exp Y$;

f') for all $A\in\exp X$, $A\subset U\subop X$ and $B\in\exp Y$
such that $B\in(AR)^\perp$, there are open sets
$V_1,\dots,V_n\subset Y$ and a~closed neighborhood $G\supset F$ in
$X$ such that $G\subset U$, $V_i\cap B\ne\ES$, $i=1,\dots,n$, and
each $B'\in GR$ contains at least one of $V_i$.
\end{stat}

\begin{proof}
We need to verify only the closedness of $R^\sms\tilde A$ for all
$\tilde A\in\exp X$. Observe that
\begin{gather*}
R^\sms\tilde A=\{\tilde B\in\exp Y\mid \forall A\in\exp X
\;\;
(\tilde B\in(AR)^\perp
\implies
A\cap\tilde A\ne\ES)
\}=
\\
\{\tilde B\in\exp Y\mid
\forall A\in\exp X
\;\;
(A\cap \tilde A=\ES
\implies
\exists B\in AR
\;\;
B\cap \tilde B=\ES)
\}.
\end{gather*}

Hence
$$
\exp Y\setminus R^\sms\tilde A=
\\
\bigcup\{(AR)^\perp\mid A\in\exp X, A\subset X\setminus \tilde A\}.
$$
The latter set is the complement to the set of all $R$-unavoidable
sets for all $A\subcl X$, $A\subset U=X\setminus \tilde A$. Thus f)
is equivalent to the closedness of $R^\sms\tilde A$ for all $\tilde
B\in\exp A$.

The equivalence of f) and f') is a particular case of
Proposition~\ref{coambl-char}, which will be proved in the next
section.
\end{proof}

\begin{cons}\label{open-comp}
The class of open ambiguous representations is closed under
composition.
\end{cons}

Thus we obtain the \emph{category of compacta and open ambiguous
representations} $\Coamb$ that is a subcategory of $\Cpamb$, and
the restriction of the contravariant functor $(-)^\sms$ to $\Coamb$
is also an involutive antiisomorphism, which preserves objects.

Now we consider order properties of the sets of ambiguous
representations. Observe that $\Camb(X,Y)$ is a lattice when
ordered by inclusion, i.e. $R$ precedes $S$ if $R\subset S$. For
$R,S\in\Camb(X,Y)$ the meet of $R,S$ is equal to $R\cap S$, and
their join is equal to $R\cup S$. The subsets $\Csamb(X,Y)$ and
$\Cpamb(X,Y)$ (and therefore $\Cpsamb(X,Y)$) are sublattices of
$\Camb(X,Y)$. To prove a similar fact about $\Coamb(X,Y)$, we need
the following statement.

\begin{stat}\label{coamb-cap-cup}
If $R,S\in\Coamb(X,Y)$, then $R\cap S, R\cup S\in\Coamb(X,Y)$.
\end{stat}

\begin{proof}
Only f) has to be checked for $R\cap S$, $R\cup S$. It is easy to
see that $C\in\exp Y$ is $R\cap S$-unavoidable for $A\in \exp X$ if
and only if $C$ is either $R$-unavoidable or $S$-unavoidable.
Therefore for any open $U\subset X$ the set of $R\cap
S$-unavoidable sets of all $A\subset U$, $A\in\exp X$, is equal to
the union of the sets of all $R$-unavoidable and of all
$S$-unavoidable sets for all $A\subset U$, $A\in\exp X$, thus it is
open in the Vietoris topology on $\exp Y$.

Similarly, $C\in\exp Y$ is $R\cup S$-unavoidable for $A\in \exp X$
if and only if $C$ is $R$-unavoidable and $S$-unavoidable. Let
$U\subop X$, $A\in\exp X$ and $B\in\exp Y$ be such that $A\subset
U$ and $B$ is $R\cup S$-unavoidable for $A$. By assumption there is
a neighborhood $OC\ni C$ in $\exp Y$ such that each $C'\in OC$ is
$R$-unavoidable for some $A_1\subset U$ and $S$-unavoidable for
some $A_2\subset U$, $A_1,A_2\in\exp X$. Then such $C'$ is $R\cap
S$-unavoidable for $A'=A_1\cup A_2\subset U$, which completes the
proof.
\end{proof}

\begin{cons}
The set $\Coamb(X,Y)$ is a sublattice of $\Cpsamb(X,Y)$.
\end{cons}

This statement can also be derived from:
\begin{stat}
Let $R,S\in\Camb(X,Y)$, then $(R\cup S)^\sms=R^\sms\cup S^\sms$,
$(R\cap S)^\sms=R^\sms\cap S^\sms$.
\end{stat}

The proof is straightforward, see also more general
Proposition~\ref{lor-land-sms}.\qed

The top and the bottom elements in the~posets $\Camb(X,Y)$, $\Csamb(X,Y)$,
and $\Cpamb(X,Y)$ are determined by the equalities:
$$
\top_{X,Y}=\exp X\times \exp Y, \;\; \bot_{X,Y}=\exp X\times
\{Y\}.
$$
Observe that $\top_{X,Y},\bot_{X,Y}$ are not always in
$\Coamb(X,Y)$.
\begin{que}
For which compacta $X,Y$ there are top and bottom elements in
$\Coamb(X,Y)$?
\end{que}
The answer is trivially positive for finite compacta.

Observe that all subsets of $\exp X\times \exp Y$ that satisfy the
definition \ref{def.sar-short} form a closed subsemilattice of the
compact Lawson upper semilattice $\exp(\exp X\times \exp Y)$, thus:

\begin{stat}
The set $\Csamb(X,Y)$ is a compact Lawson upper semilattice.
\end{stat}

Observe that for non-finite $X,Y$ the lattice $\Csamb(X,Y)$ is not
topological, for meet (=intersection) is not continuous in general
w.r.t.\ the Vietoris topology.

\begin{que}
What are topological properties of the subsets $\Cpsamb(X,Y)$,
$\Coamb(X,Y)$ of $\Csamb(X,Y)$?
\end{que}

For the operation $(-)^\sms$ is involutive and isotonic, it
provides isomorphisms of lattices $\Cpamb(X,Y)\cong\Cpamb(Y,X)$ and
$\Coamb(X,Y)\cong\Coamb(Y,X)$ for all compacta $X,Y$.

Recall that an \emph{allegory} \cite{Winter:CatApprToLFuzzyRel:07}
is a category $\CCC$ in which:

1) for all objects $X,Y$ each set $\CCC(X,Y)$ is a lower
semilattice; we denote its meet (also called \emph{intersection} in
this case) and order by $\land$ and $\prec$, respectively;

2) for all objects $X,Y$ there is a monotonic operation $(-)^\sms$
(\emph{anti-involution} or \emph{converse} operation), that takes
every morphism $f:X\to Y$ to a morphism $f^\sms:Y\to X$, such that
$(f^\sms)^\sms=f$ and $(f\circ g)^\sms=g^\sms\circ f^\sms$,
provided that $f\circ g$ exists;

3) composition is monotone in both arguments; and

4) the \emph{modular law} holds: if $f:X\to Y$, $g:Y\to Z$, $h:X\to
Z$, then $g\circ f\land h\prec (g\land h\circ f^\sms)\circ f$.

We can see that $\Cpamb$ and $\Coamb$ with the defined operation
$(-)^\sms$ satisfy all the requirements of the definition of
allegory but the last (which is not too surprising, because we
generalize mappings rather than relations). Moreover, composition
is distributive over meet and join in the both arguments.



\section{$L$-ambiguous representations}

In the sequel $L$ will be a compact Lawson lattice, $0$ and $1$ the
bottom and the top elements of $L$. To fit into a~described in
the~introductory section $L$-fuzzy framework, we consider
an~operation $*:L\times L\to L$, which is associative, commutative,
isotone in the both arguments and $1$ is a neutral element for
"$*$". We demand that $*$ be lower semicontinuous and distributive
w.r.t. $\lor$ in the both arguments. Due to compactness this
implies infinite distributive laws. Sometimes we shall need also
the~upper semicontinuity, i.e.\ the~continuity of~$*$. The simplest
such "$*$" is the lattice meet "$\land$". See also
\cite{Dross:GenTNormStr:99} for more information on such
operations.

For each subset $R\subset X\times Y\times L$ (not only for
a~subgraph of an $L$-relation) and $\alpha\in L$, we define the
$\alpha$-\emph{cut} $R_\alpha$ as follows
$$
R_\alpha =
\{(a,b)\in X\times Y\mid (a,b,\alpha)\in R\}.
$$
For subsets $A\subset X$, $B\subset Y$ we put
\begin{gather*}
 AR=\{(y,\alpha)\in Y\times L\mid \text{there is }x\in A
\text{ such that }(x,y,\alpha)\in R\},
\\
RB=\{(x,\alpha)\in X\times L\mid \text{there is }y\in B
\text{ such that }(x,y,\alpha)\in R\}.
\end{gather*}

\begin{defn}
Let $X,Y$ be compacta, $L$ be a compact Lawson lattice. A subset
$R\subset \exp X\times \exp Y\times L$ is called an
$L$-\emph{ambiguous representation} between $X$ and $Y$ if:

a) if $A,A'\in\exp X$, $B,B'\in \exp Y$, $\alpha,\alpha'\in L$,
$A'\subset A$, $B\subset B'$, $\alpha\ge\alpha'$, then
$(A,B,\alpha)\in R$ implies $(A',B',\alpha')\in R$;

b) if $A\in\exp X$, $B\in\exp Y$, $\alpha,\beta\in L$ are such that
$(A,B,\alpha),(A,B,\beta)\in R$, then $(A,B,\alpha\lor\beta)
\in R$;

c) $(A,Y,\alpha),(A,B,0)\in R$ for all $A\in\exp X$, $B\in\exp Y$,
$\alpha\in L$; and

d) for all $A\in\exp X$ the set $AR=\{(B,\alpha)\in\exp Y\times
L\mid (A,B,\alpha)\in R\}$ is closed in $\exp Y\times L$.
\end{defn}

It is equivalent to $R$ being a~subgraph of an $L$-fuzzy binary
relation between $\exp X$ and $\exp Y$ (denoted by the~same letter
$R$ for brevity) such that:

a') $R$ is antitone in the first argument;

b') $R$ is isotone and upper semicontinuous (i.e.\ preserves
filtered infima) in the second argument;

c') $R(A,Y)=1$ for all $A\in \exp X$.

The~mapping $R:\exp X\times\exp Y\to L$ is uniquely recovered by
the formula $R(A,B)=\max\{\alpha\in L\mid (A,B,\alpha)\in R\}$.
The~value $R(A,B)\in L$ is interpreted as a~degree which shows how
well $A$ can represent $B$ (the~more, the~better). If
$(A,B,\alpha)\in R$, then $A$ represents $B$ with fitness at~least
$\alpha$. We shall interchange the~relational and the~functional
interpretations of $L$-ambiguous representations, whatever is more
convenient in a~particular case.

\begin{exam}
Let $(X,d)$ be a~metric compactum and $L=[0,1]$. We put
\begin{multline*}
R(A,B)= 1-\sup\{d(a,B)\mid a\in A\}/\diam X,
\;
A\in\exp X,B\in\exp Y.
\end{multline*}
Then $R(A,B)\ge\alpha$ if and only if $A$ extends beyond $B$ by no
more that $\delta=(1-\alpha)\diam X$. It means that $A$ can be
obtained from a~closed subset $B_0\subset B$ by ``shifts'' of its
points by $\le\delta$ in different (probably multiple) directions.
\end{exam}

\begin{exam}
Let $X\subset Y$ be compact subsets of $\BBR^2$, $\diam Y=r$,
$L=[0,r]$. For $A\in\exp Y$, $B\in \exp Y$, we put
$$
R(A,B)=r-\inf\{\|\vec m\|\mid
\vec m\in\BBR^2,A+\vec m\subset B\},
$$
assuming $\inf\ES=r$. Then $R(A,B)\ge\alpha>0$ iff $A$ can be
shifted in \emph{one} direction by a~distance $\le r-\alpha$ to
coincide with a~subset of $B$.
\end{exam}

The~two latter examples implement a~common idea: $R(A,B)$ shows how
well a~subset $A$ represents a~part of an~image of $B$. Of course, we
can combine shifts, expansions, rotations etc, depending on which
distortions of images are expected. It is also not necessary that
$X$ and $Y$ be of equal dimension. Assume, e.g., that $Y$ is
an~area in atmosphere, and $X$ is a~finite set of points on
the~earth surface where automatic registration devices are
installed. If there is a~snow cloud somewhere in $Y$, and $A\subset
X$ is a~(probably incomplete) set of points where snowfall is
observed, then it is possible to define a~function $R(A,B)$, which
will estimate the~likelihood that the cloud is contained in
a~subset $B\subset Y$. Such problems, where uncertainty,
distortions and incompleteness of information combine, are the~main
target of the~introduced $L$-fuzzy ambiguous representations.

Now we discuss how they are related to concepts used in fuzzy sets
theory, in particular, in fuzzy topology. The~latter theory in its
different flavors~\cite{Sho:2decFuzzyTop} studies crisp or fuzzy
families of \emph{fuzzy subsets} of a~universe. We are not going so
far in ``fuzzification'', and only crisp or fuzzy \emph{relations}
between hyperspaces of \emph{ordinary} closed subsets of compacta
are considered, although ``totally fuzzy'' generalizations of
ambiguous representations can also be introduced. Probably, to
develop a~consistent theory, these future extensions would require
use of sheaf-theoretic
apparatus~\cite{Hoh:FuzzySetsSheavesI:07,Hoh:FuzzySetsSheavesII:07}.
Hence, similarly to fuzzy topology in its ``more fuzzy'' variants,
different degrees of membership of a~set in a~family of valid
representatives for another set can occur. We see no reasons to
restrict ourselves to the~unit interval to express membership, and
prefer Goguen's lattice-valued
approach~\cite{Goguen:LFuzzySets:67}. Note that we use lattice
elements to describe rather \emph{quality} of representations,
which does not necessarily relates to probabilistic interpretation
of fuzzy sets.

Of course, even this ``moderate'' $L$-fuzziness of our
constructions inevitably leads to ``graded'' families similar to
studied by Negoita and Ralescu~\cite{NegRal:RepTheo:75}. Observe
that our level cuts are not sets of individual points, but
relations between hyperspaces.

The set of all $L$-ambiguous representations between $X$ and $Y$ is
denoted by $\Camb_L(X,Y)$.

\begin{defn}
An $L$-ambiguous representation $R\subset \exp X\times \exp Y\times
L$ is \emph{strict} if for all $B\in \exp Y$ the set
$RB=\{(A,\alpha)\in\exp X\times L\mid (A,B,\alpha)\in R\}$ is
closed in $\exp X\times L$.
\end{defn}

We denote the set of all strict $L$-ambiguous representations
between $X$ and $Y$ by $\Csamb_L(X,Y)$.

By the following lemma a strict $L$-ambiguous representation
$R\subset \exp X\times \exp Y\times L$ is a closed subset.

\begin{lem}\label{int-un-sup-cl}
Let $X,Y$ be compacta, $L$ a compact Lawson upper semilattice, and
let a subset $R\subset\exp X\times\exp Y\times L$ be such that, for
$A,A'\in\exp X$, $B,B'\in \exp Y$, $\alpha,\alpha'\in L$,
$A'\subset A$, $B\subset B'$, $\alpha'\le \alpha$, if
$(A,B,\alpha)\in R$, then $(A',B',\alpha')\in R$. Then $R$ is
closed if and only if the following two conditions hold:

1) for all $A\in\exp X$, $\alpha\in L$ and each filtered collection
$\CCB$ of elements of $\exp Y$ such that
$\{A\}\times\CCB\times\{\alpha\}\subset R$, we have
$(A,\bigcap\CCB,\alpha)\in R$; and

2) for all $B\in\exp Y$ the set of all $(A,\alpha)\in\exp X\times
L$ such that $(A,B,\alpha)\in R$ is closed.
\end{lem}

This can be derived from Lemma~\ref{int-un-cl} by a simple
observation that $\exp X\times L$ is a compact Lawson upper
semilattice.

For $L$-ambiguous representations $R\subset \exp X\times
\exp Y\times L$, $S\subset \exp Y\times \exp Z\times L$ we define
the composition $R\acirc S$ in the following manner, which is
customary for $L$-relations (cf.\ the~introductory section):

\begin{multline*}
R\acirc S=
\bigl\{(A,C,\alpha)\in \exp X\times \exp Z\times L
\mid
\alpha\le\sup
\{\beta*\gamma\mid
\\
\text{there is }B\in \exp Y
\text{ such that }(A,B,\beta)\in R,(B,C,\gamma)\in S
\}
\bigr\},
\end{multline*}
or, equivalently, in the~functional notation:
$$
R\acirc S(A,C)=
\sup\bigl\{R(A,B)*S(B,C)
\mid
B\in\exp Y
\bigr\},
$$
for $A\in\exp X$, $C\in\exp Z$.

\begin{stat}
If $*:L\times L\to L$ is continuous, $R\in\Csamb_L(X,Y)$, and
$S\in\Csamb_L(Y,Z)$, then $R\acirc S$ is a strict $L$-ambiguous
representation.
\end{stat}

\begin{proof}
The set
\begin{multline*}
R\adot S=\{(A,C,\alpha)\in \exp X\times \exp Z\times L\mid
\\
B\in\exp Y,\;\beta,\gamma\in L,\;\alpha\le\beta*\gamma,\;
(A,B,\beta)\in R, (B,C,\gamma)\in S\}
\end{multline*}
is closed in $\exp X\times \exp Z\times L$. Therefore the set of
all subsets of $R\adot S\subset \exp X\times \exp Y\times L$ of the
form $\{A\}\times\{B\}\times M$, with $M\subcl L$, $M\ne\ES$, is
closed in $\exp(\exp X\times \exp Y\times L)$, hence its image
under the continuous correspondence that takes each
$\{A\}\times\{B\}\times M$ to $(A,B,\sup M)$ is closed. This image
is equal to $R\acirc S$. Other properties are obvious.
\end{proof}

Hence we obtain the~\emph{category of compacta and strict
$L$-ambiguous representations} $\Csamb^*_L$. The~composition of
sequential arrows $R:X\to Y$ and $S:Y\to Z$ in this category, i.e.\
of $R\in\Csamb_L(X,Y)$, $S\in\Csamb_L(Y,Z)$, is equal to $R\acirc
S$. For a compactum $X$, the identity morphism in this category is
equal to
$$
\uni{X}=\{(A,B,\alpha)\in \exp X\times \exp X\mid A\subset B
\text{ or }\alpha=0\}.
$$

If $*$ is not continuous or $L$-ambiguous representations are not
strict, then their composition is not necessarily an~$L$-ambiguous
representation. Therefore we must repeat the trick which was used
for the crisp case. For $L$-ambiguous representations $R\subset
\exp X\times\exp Y\times L$, $S\subset\exp Y\times \exp Z\times L$,
let $A(R\bacirc S)=\Cl(A(R\acirc S))$ for all $A\in\exp X$. In
other words, $(C,\gamma)\in A(R\bacirc S)$ if and only if for all
$\gamma'\ll\gamma$ and each closed neighborhood $V\supset C$ there
are $B_1,\dots,B_n\in\exp Y$, $\alpha_1,\beta_1,
\dots,\alpha_n,\beta_n\in L$ such that
\begin{gather*}
(A,B_1,\alpha_1),\dots,(A,B_n,\alpha_n)\in R,
(B_1,V,\beta_1),\dots,(B_n,V,\beta_n)\in S,
\\
\alpha_1*\beta_1\lor\dots\lor\alpha_n*\beta_n\ge\gamma'.
\end{gather*}
If $*=\land$, then we write $\bcirc$ for $\bacirc$. If only
strict $L$-ambiguous representations are taken, then
$\bacirc=\acirc$.

Although $\acirc$ is associative, the composition $\bacirc$ of
$L$-ambiguous representations is not associative in the general
case. Thus we must impose further restrictions on the class of
allowed relations.

For a relation $R\subset\exp X\times\exp Y\times L$ such that all
its $\alpha$-cuts are ambiguous representations, we define a
relation $R^\sms\subset\exp Y\times
\exp X\times L$ by the equality
$(R^\sms)_\alpha=\bigcap_{\beta\ll\alpha}(R_\beta)^\sms$. In
other words, $(B,A,\alpha)\in R^\sms$ if and only if the set $A$
has non-empty intersections with all $A'\in\exp X$ such that $B$ is
$R_\beta$-unavoidable for $A'$ for some $\beta\ll\alpha$.

\begin{stat}
If $R\subset \exp X\times\exp Y\times L$ is an $L$-ambiguous
representation, then so is $R^\sms$.
\end{stat}

\begin{proof}
It is obvious that $R^\sms\subset\exp Y\times \exp X\times L$ is
closed, contains $\exp Y\times \exp X\times \{0\}\cup \exp Y\times
\{X\}\times L$, and $(B,A,\alpha)\in R^\sms$, $B\supset B'\in\exp
Y$, $A\subset A'\in\exp X$, $\alpha\ge\alpha\in L$ implies
$(B',A',\alpha')\in R^\sms$.

Let $(A,B,\alpha),(A,B,\beta)\in R^\sms$, then $A\cap A'\ne\ES$ for
all $A'\in\exp X$ such that $B\in(A'R_{\alpha'})^\perp$ for some
$\alpha'\ll\alpha$ or $B\in(A'R_{\beta'})^\perp$ for some
$\beta'\ll\beta$. Let $\gamma\ll\alpha\lor\beta$, then
due to Lemma~\ref{lem.ldot} there are $\alpha'\ll\alpha$,
$\beta'\ll\beta$ such that $\alpha'\lor\beta'\ge\gamma$. Then
$$
A'R_{\gamma}\supset A'R_{\alpha'\lor\beta'}=A'R_{\alpha'}\cap
A'R_{\beta'}.
$$

Hence
$$
A'R_{\gamma}^\perp
\subset
(A'R_{\alpha'}\cap A'R_{\beta'})^\perp
=
(A'R_{\alpha'})^\perp\cup (A'R_{\beta'})^\perp,
$$
therefore $B\in A'R_{\gamma}^\perp$ implies $A'\cap A\ne\ES$, i.e.
$(A,B,\alpha\lor\beta)\in R^\sms$. Thus $R^\sms$ is an
$L$-ambiguous representation.
\end{proof}

\begin{lem}\label{lem.rs-sms}
For $L$-ambiguous representations $R\subset \exp X\times\exp
Y\times L$, $S\subset\exp Y\times \exp Z\times L$ the inclusion
$S^\sms\bacirc R^\sms\subset(R\bacirc S)^\sms$ is valid.
\end{lem}

\begin{proof}
Let $(C,A,\gamma)\in (S^\sms)\bacirc (R^\sms)$, then for all closed
neighborhoods $V\supset A$ and all $\gamma'\ll\gamma$ there
are $B_1,\dots,B_n\in\exp Y$, $\alpha_1,\beta_1,
\dots,\alpha_n,\beta_n\in L$ such that
\begin{gather*}
(C,B_1,\alpha_1),\dots,(C,B_n,\alpha_n)\in S^\sms,
(B_1,V,\beta_1),\dots,(B_n,V,\beta_n)\in R^\sms,
\\
\alpha_1*\beta_1\lor\dots\lor\alpha_n*\beta_n\ge\gamma'.
\end{gather*}
Given an element $\delta\in L$ such that $\delta\ll \gamma$,
we choose $\gamma'\in L$ such that $\delta\ll\gamma'\ll
\gamma$.

For all $A'\in\exp X$ such that $A'\cap V=\ES$, and all
$i\in\{1,\dots,n\}$, $\beta_i'\ll \beta_i$, there is $B'_i\in\exp
Y$ such that $B'_i\cap B_i=\ES$, $(A',B'_i,\beta_i')\in R$.
Similarly, for all $i\in\{1,\dots,n\}$, $\alpha_i'\ll \alpha_i$,
there is $C'_i\in\exp Y$ such that $C'_i\cap C=\ES$,
$(B'_i,C'_i,\alpha_i')\in S$. Due to the continuity of $\lor$ and
the lower semicontinuity of $*$, we can choose $\alpha_i',\beta_i'$
so that
$$
\alpha'_1*\beta'_1\lor\dots\lor\alpha'_n*\beta'_n\ge\delta.
$$
Then the set $C'=C_1'\cup\dots\cup C_n'$ is closed and nonempty,
and $C'\cap C=\ES$, $(A',C',\delta)\in R\bacirc S$. Such $C'$
exists for all $\delta\ll\gamma$ and all $A'\in\exp X$ such
that $A'\cap V=\ES$ for some closed neighborhood $V\supset A$, i.e.
for all $A'\in\exp X$ such that $A'\cap A=\ES$. Thus
$(C,A,\gamma)\in(R\bacirc S)^\sms$.
\end{proof}

\begin{stat}\label{char-lpseudo}
For an $L$-ambiguous representation $R\subset \exp X\times \exp
Y\times L$, the inclusion $(R^\sms)^\sms\subset R$ is valid, and
$(R^\sms)^\sms= R$ if and only if for all $(A,B,\alpha)\in R$,
$\beta\in L$ such that $\beta\ll\alpha$, and a closed
neighborhood $V\supset B$, there is a closed neighborhood $U\supset
A$ such that $(U,V,\beta)\in R$.
\end{stat}

\begin{proof}
By the definition, for all $B\in\exp Y$:
\begin{gather*}
BR^\sms_\alpha=
\bigcap_{\beta\ll\alpha}
\{A\in\exp X\mid B\in (AR_\beta)^\perp\}^\perp=
\\
(\bigcup_{\beta\ll\alpha}
\{A\in\exp X\mid B\in (AR_\beta)^\perp\})^\perp,
\end{gather*}
hence for all $\tilde A\in\exp X$:
\begin{gather*}
A(R^\sms)^\sms_\alpha= (\bigcup_{\beta\ll\alpha}
\{B\in\exp X\mid A\in (BR^\sms_\beta)^\perp\})^\perp=
\\
(\bigcup_{\beta\ll\alpha}
\{B\in\exp X\mid A\in
(\bigcup_{\gamma\ll\beta}
\{A'\in\exp X\mid B\in (A'R_\gamma)^\perp\}
)^{\perp\perp}
\}
)^\perp=
\\
(\bigcup_{\beta\ll\alpha}
\{B\in\exp X\mid A\in
\Cl(\bigcup_{\gamma\ll\beta}
\{A'\in\exp X\mid B\in (A'R_\gamma)^\perp\}
)\}
)^\perp=
\\
(\bigcup_{\beta\ll\alpha}
\{B\in\exp X\mid
\text{ for all }U\in\exp X,
A\subset \Int U
\\
\text{ there is }\gamma\ll\beta\text{ such that }
B\in(UR_\gamma)^\perp
\}
)\}
)^\perp=
\\
(\bigcup_{\beta\ll\alpha}
\{B\in\exp X\mid
B\in(UR_\beta)^\perp
\text{ for all }U\in\exp X,
A\subset \Int U
\}
)\}
)^\perp\subset AR_\alpha.
\end{gather*}

The equality $(A^\sms)^\sms_\alpha=AR_\alpha$ is
equivalent to:
$$
\bigcup_{\beta\ll\alpha}
\{B\in\exp X\mid
B\in(UR_\beta)^\perp
\text{ for all }U\in\exp X,
A\subset \Int U
\}
)\}=(AR_\alpha)^\perp.
$$

It fails if and only if there is $\tilde B\in\exp Y$ such that
$\tilde B\in(UR_\beta)^\perp$ for all $\beta\ll\alpha$,
$U\in\exp X$, $A\subset \Int U$, but $\tilde B\notin
(AR_\alpha)^\perp$, i.e. there is $B\in\exp Y$ such that
$B\cap\tilde B=\ES$, $(A,B,\alpha)\in R$. In this case let $V$ be a
closed neighborhood of $B$ such that $V\cap \tilde B=\ES$, then
$V\notin UR_\beta$ for all $\beta\ll\alpha$ and closed
neighborhoods $U$ of $A$. Thus the condition of the proposition is
sufficient for the equality $(R^\sms)^\sms=R$. On the contrary, let
such $A$, $B$, $V$ and $\alpha$ exist, then $\tilde B=Y\setminus
\Int V$ is a required counterexample, and the condition is
necessary.
\end{proof}

\begin{defn}
If $R\subset \exp X\times \exp Y\times L$ is an $L$-ambiguous
representation such that $(R^\sms)^\sms=R$, then we call $R^\sms$
\emph{pseudo-inverse} to $R$, and $R$ is called
\emph{pseudo-invertible}.
\end{defn}

We denote the set of pseudo-invertible $L$-ambiguous
representations from $X$ to $Y$ by $\Cpamb_L(X,Y)$.

By Lemma~\ref{lem.rs-sms}, similarly to
Proposition~\ref{cpamb-comp} we obtain:

\begin{stat}\label{cpambl-comp}
Let $R\subset\exp X\times \exp Y\times L$ and $S\subset \exp
Y\times\exp Z\times L$ be pseudo-invertible $L$-ambiguous
representations. Then $R\bacirc S\subset\exp X\times
\exp Z\times L$ is a pseudo-invertible $L$-ambiguous representation as well, and
$(R\bacirc S)^\sms= S^\sms\bacirc R^\sms$.
\end{stat}

\begin{stat}
For compacta $X,Y$ and a compact Lawson lattice $L$, a relation
$R\subset \exp X\times \exp Y\times L$ is an ambiguous
representation if and only if for all $A\in\exp X$ the set
$AR\subset \exp Y\times L$ is the subgraph of an $L$-capacity
$c_{AR}\in M_LY$, and the correspondence $A\mapsto c_{AR}$ is an
antitone mapping from $\exp X$ to $M_LY$. This mapping is:

$\bullet$ upper semicontinuous if and only if $R$ is strict;

$\bullet$ lower semicontinuous if and only if $R$ is
pseudo-invertible;

$\bullet$ continuous if and only if $R$ is pseudo-invertible and
strict.
\end{stat}
\begin{defn}
An $L$-ambiguous representation $R$ is called
\emph{open} if both $R$ and $R^\sms$ are strict $L$-ambiguous
representations, and $(R^\sms)^\sms=R$.
\end{defn}

\begin{stat}\label{coambl-char}
A pseudo-invertible strict $L$-ambiguous representation
$R\subset\exp X\times \exp Y\times L$ is open if and only if
any of the
following statements is valid:

$\bullet$ for any open $U\subset X$ and all $\alpha\in L$ the set
of $R_\beta$-unavoidable sets of all $A\subset U$, $A\in\exp X$,
for all $\beta\in L$, $\beta\ll \alpha$, is open in the
Vietoris topology on $\exp Y$;

$\bullet$ for all $A\in\exp X$, $A\subset U\subop X$, $B\in\exp Y$
and $\alpha,\beta\in L$ such that $B\in(AR_\beta)^\perp$,
$\beta\ll\alpha$, there are open sets $V_1,\dots,V_n\subset
Y$, $\gamma\in L$ and a~\emph{closed} neighborhood $G\supset A$ in
$X$ such that $G\subset U$, $\gamma\ll\alpha$, $V_i\cap
B\ne\ES$, $i=1,\dots,n$, and each $B'\in GR_\gamma$ contains at
least one $V_i$.
\end{stat}

\begin{proof}
To prove that the first statement is equivalent to $R$ being open,
it is sufficient to observe that for each $\tilde A\in\exp X$ the
complement $\exp Y\setminus R^\sms\tilde A$ is equal to the set of
$R_\beta$-unavoidable sets for all $A\in\exp X$, $A\subset
U=X\setminus \tilde A$, $\beta\ll\alpha$.

Let $R$ be an open $L$-ambiguous representation, hence $R^\sms$ is
a closed subset, and let $A\in\exp X$ be such that
$B\in(AR_\beta)^\perp$ for some $\beta\ll\alpha$. If we take
$A_0=X\setminus U$, then $(A_0,B,\alpha)\notin R^\sms$, hence there
must exist neighborhoods $U'\supset A_0$ in $X$, $\langle
V_1,\dots, V_n\rangle\ni B$ in $\exp Y$, $W_\alpha\ni\alpha$ in $L$
such that, for all $A'\in\exp X$, $A'\subset U'$, $B'\in\langle
V_1,\dots, V_n\rangle$, $\alpha'\in W_\alpha$, there is $A''\in\exp
X$ such that $A'\cap A''=\ES$ and $B'$ is $R_{\beta'}$-unavoidable
for $A''$ for some $\beta'\ll\alpha'$. If necessary, we can
make $V_i$ smaller to be disjoint. We can also choose $U'$ so that
$\Cl U'\cap A=\ES$. Moreover, $B'$ is $R_{\beta'}$-unavoidable for
$G=X\setminus U'$, and $G$ is a closed neighborhood of $A$. The
lattice $L$ is compact Lawson, hence there is an open neighborhood
$O_\alpha\ni\alpha$, $O_\alpha\subset W_\alpha$, such that $\inf
O_\alpha\in W_\alpha$. Let $\gamma=\inf O_\alpha$, then
$\gamma\ll\alpha$, and each $B'\in\langle V_1,\dots,
V_n\rangle$ is $R_{\gamma}$-unavoidable for $G$. It is possible if
and only if each element of $GR_{\gamma}$ is a superset of some of
$V_i$. Necessity of the second statement is proved.

Now the proof of its sufficiency is obvious.
\end{proof}

\begin{stat}\label{coambl-comp}
Let $*:L\times L\to L$ be open, $R\subset\exp X\times \exp Y\times
L$ and $S\subset\exp Y\times\exp Z\times L$ open $L$-ambiguous
representations. Then $R\acirc S\subset\exp X\times
\exp Z\times L$ is an open $L$-ambiguous representation as well.
\end{stat}

\begin{proof}
By Proposition~\ref{cpambl-comp}, the relation $(R\acirc S)^\sms$
is equal to the composition of two strict $L$-ambiguous
representations $S^\sms$ and $R^\sms$, hence is a strict
$L$-ambiguous representation itself.
\end{proof}

By Proposition~\ref{char-lpseudo} the composition of
pseudo-invertible $L$-ambiguous representations is associative.
Thus we obtain a collection of the \emph{categories of compacta and
pseudo-invertible $L$-ambiguous representations} $\Cpamb^*_L$ with
the same objects (=compacta) and morphisms (=pseudo-invertible
$L$-ambiguous representations), but with different laws of
composition $\bacirc$ parameterized by certain t-norms on the
lattice $L$. For all of them there is an involutive isomorphism
$(-)^\sms:(\Cpamb^*_L)^{op}\to\Cpamb^*_L$ which preserves objects.
For $*$ continuous, each of these categories contains a subcategory
$\Coamb^*_L$ with all \emph{open $L$-ambiguous representations} as
morphisms and the~composition law $\bacirc=\acirc$. We also denote
the intersection of $\Cpamb^*_L$ and $\Csamb^*_L$ by $\Cpsamb^*_L$.
The identity arrows for these categories are the same as in
$\Csamb^*_L$.

Again, if $*=\land$, we omit it in the notation for categories.

Each ambiguous representation $R\subset \exp X\times \exp Y$ can be
identified with an $L$-ambiguous representation $R_L$ defined as
follows:
$$
R_L=\{(A,B,\alpha)\in \exp X\times \exp Y\mid (A,B)\in R
\text{ or }\alpha=0\}.
$$
Then the categories of $\Csamb$, $\Cpamb$, $\Cpsamb$, and $\Coamb$
are embedded into the respective categories of $\Csamb^*_L$,
$\Cpamb^*_L$, $\Cpsamb^*_L$, and $\Coamb^*_L$ (independently of $*$).

When we attempt to study order and topological properties of sets
of (strict, open) $L$-ambiguous representations in the same manner
as we did before for (non-fuzzy) representations, we encounter new
difficulties. If $\Camb_L(X,Y)$ is ordered by inclusion, then the
top and the bottom elements of this poset are obvious:
$$
\top_{X,Y,L}=\exp X\times \exp Y\times L, \;\;
\bot_{X,Y,L}=\exp X\times\{Y\}\times L\cup
\exp X\times \exp Y\times\{0\}.
$$
If $R,S\in\Camb_L(X,Y)$, then $R\cap S\in\Camb(X,Y)$, but, for
$|X|>1$, $|Y|>1$ and a non-linearly ordered $L$, not always $R\cup
S\in\Camb_L(X,Y)$. E.g. let $\alpha,\beta\in L$ be incomparable,
$x_1,x_2\in X$, $y\in Y$ and
\begin{gather*}
R=
\bot_{X,Y,L}\cup
\{(\{x_1\},F,\gamma)
\mid
y\in F\in\exp Y, \gamma\in L, 0\le \gamma\le\alpha
\}
,\\
S=
\bot_{X,Y,L}\cup
\{(\{x_2\},F,\gamma)
\mid
y\in F\in\exp Y, \gamma\in L, 0\le \gamma\le\beta
\},
\end{gather*}
then $R$ and $S$ are $L$-ambiguous representations, but, if
$x_1=x_2$, then $R\cup S$ is not. This implies that, for an
infinite $X$, $|Y|>1$, and a non-linearly ordered $L$, the set
$\Csamb_L(X,Y)$ is not closed in $\exp(\exp X\times\exp Y\times
L)$, although its elements are closed sets.

It is easy to describe suprema and infima in $\Camb_L(X,Y)$ and
$\Csamb_L(X,Y)$. For a set $\CCR\subset\Camb_L(X,Y)$ its lowest
upper bound is a relation $R_0\subset \exp X\times\exp Y\times L$
defined by the equality
\begin{gather*}
AR_0=
\Cl\bigl\{(B,\gamma)\in\exp Y\times L
\mid
\gamma\le\sup\{\alpha\in L\mid
\\
(A,B,\alpha)\in R
\text{ for some }R\in\CCR\}\bigr\}
\end{gather*}
for all $A\in\exp X$.

Similarly, if a strict $L$-ambiguous representation $S$ is an upper
bound of a subset $\CCR$ in $\Csamb_L(X,Y)$, then $S$ must contain
a set
\begin{gather*}
R_0=
\Cl\bigl\{(A,B,\gamma)\in\exp X\times\exp Y\times L
\mid
\gamma\le\sup\{\alpha\in L\mid
\\
(A,B,\alpha)\in R
\text{ for some }R\in\CCR\}\bigr\}.
\end{gather*}
It is obvious that $R_0$ is a strict $L$-ambiguous representation
and a least upper bound of $\CCR$.

The greatest lower bound of a subset $\CCR$ of $\Camb_L(X,Y)$ or
$\Csamb_L(X,Y)$ is simply the intersection of $\CCR$.

In both these sets the pairwise supremum of $R_1$, $R_2$ is
determined by the equality
$$
R_1\lor R_2=
\{(A,B,\alpha_1\lor\alpha_2)
\mid
(A,B,\alpha_1)\in R_1, (A,B,\alpha_2)\in R_2\}.
$$

\begin{stat}\label{lor-land-sms}
Let $R,S\in\Camb_L(X,Y)$, then $(R\lor S)^\sms=R^\sms\lor S^\sms$,
$(R\land S)^\sms=R^\sms\land S^\sms$.
\end{stat}

\begin{proof}
The operation $(-)^\sms$ is isotone, hence $(R\lor S)^\sms\supset
R^\sms\lor S^\sms$, $(R\land S)^\sms\subset R^\sms\land S^\sms$.

Let $(B,A,\gamma)\notin R^\sms\lor S^\sms$, i.e. for all
$\alpha,\beta\in L$ such that $\alpha\lor\beta\ge\gamma$, we have
either $(B,A,\alpha)\notin R^\sms$ or $(B,A,\beta)\notin S^\sms$,
i.e. either there are $A'\in\exp X$, $\alpha'\ll \alpha$ such
that $A'\cap A=\ES$, $B\in A'R_{\alpha'}$, or there are $A'\in\exp
X$, $\beta'\ll
\beta$ such that $A'\cap A=\ES$, $B\in A'S_{\beta'}$. The set
$\{(\alpha,\beta)\in L^2\mid
\alpha\lor\beta \ge \gamma\}$ is compact, therefore there is a finite
collection $\alpha_1,\dots,\alpha_m,\beta_1,\dots,\beta_n\in L$ and
a closed nonempty set $A'\subset X$ such that $A'\cap A=\ES$,
$$
B\in (A'R_{\alpha_1})^\perp\cap\dots\cap (A'R_{\alpha_m})^\perp\cap
(A'S_{\beta_1})^\perp\cap\dots\cap (A'S_{\beta_n})^\perp,
$$
and, for all $\alpha,\beta\in L$ such that
$\alpha\lor\beta\ge\gamma$, either $\alpha_i\ll \alpha$ for
some $1\le i\le m$, or $\beta_j\ll \beta$ for some $1\le j\le
n$. Hence there is $\gamma'\ll \gamma$ such that, for all
$\alpha,\beta\in L$ such that $\alpha\lor\beta\ge\gamma'$, either
$\alpha_i\ll
\alpha$ for some $1\le i\le m$, or $\beta_j\ll \beta$ for some
$1\le j\le n$. Then $(R\lor S)_{\gamma'}\subset R_{\alpha_1}\cup
\dots \cup R_{\alpha_m}\cup S_{\beta_1}\cup
\dots \cup S_{\beta_m}$, and
$$
(A'(R\lor S)_{\gamma'})^\perp
\supset
(A'R_{\alpha_1})^\perp\cap\dots\cap (A'R_{\alpha_m})^\perp\cap
(A'S_{\beta_1})^\perp\cap\dots\cap (A'S_{\beta_n})^\perp,
$$
therefore $B\in (A'(R\lor S)_{\gamma'})^\perp$ for some
$\gamma'\ll \gamma$, $A'\in\exp X$, $A'\cap A=\ES$, thus
$(B,A,\gamma)\notin (R\lor S)^\sms$. We have proved that
$R^\sms\lor S^\sms=(R\lor S)^\sms$.

Let $(B,A,\gamma)\in R^\sms\land S^\sms$, then for all
$\gamma'\ll \gamma$ and $A'\in\exp X$, $A'\cap A=\ES$, there
are $B_1\in A'R_{\gamma'}$ and $B_2\in A'S_{\gamma'}$ such that
$B_1\cap B=B_2\cap B=\ES$. Then $B'=B_1\cup B_2\in A'(R\land
S)_{\gamma'}$, $B'\cap B=\ES$, hence $(B,A,\gamma)\in (R\land
S)^\sms$. The equality $R^\sms\land S^\sms=(R\land S)^\sms$ is also
proved.
\end{proof}

\begin{cons}\label{coambl-cap-sup}
If $R,S$ are elements of $\Cpamb_L(X,Y)$ (or $\Cpsamb_L(X,Y)$, or
$\Coamb_L(X,Y)$), then $R\lor S$ and $R\land S$ are also in
$\Cpamb_L(X,Y)$ (resp.\ in $\Cpsamb_L(X,Y)$ or $\Coamb_L(X,Y)$).
\end{cons}

Thus $\Cpamb_L(X,Y)$, $\Cpsamb_L(X,Y)$ and $\Coamb_L(X,Y)$ are
\emph{sublattices} of the lattice $\Camb_L(X,Y)$.

To address topological issues, we define for each
$R\in\Csamb_L(X,Y)$ a relation $R^\cup\subset
\exp^2X\times \exp Y\times L$ by the equality
\begin{gather*}
R^\cup=
\bigl\{
(\CCA,B,\gamma)
\mid
\CCA\in\exp^2X,B\in\exp Y,\gamma\in L,
\\
\gamma\le
\sup\{\alpha\in L\mid
(A,B,\alpha)\in R\text{ for some }A\in \CCA\}
\bigr\}.
\end{gather*}

Observe that for $A\in\exp X$, $B\in\exp Y$, $\alpha\in L$ and
$R\in\Csamb_L(X,Y)$ the inclusions $(A,B,\alpha)\in R$ and
$(\{A\},B,\alpha)\in R^\cup$ are equivalent, therefore the
correspondence $R\mapsto R^\cup$ is injective. Let us consider the
image of this correspondence.

\begin{stat}
The relations $R^\cup$ for all $R\in \Csamb_L(X,Y)$ are closed sets
and form a~closed subset of $\exp(\exp^2X\times \exp Y\times L)$.
\end{stat}

In order to prove this proposition, we define an operation on
$\exp(\exp^2X\times \exp Y\times L)$ such that all sets of the form
$R^\cup$ are fixed points of this operation.

For a compactum $X$ and a closed non-empty set $\CCA\subset \exp X$, the set
$$
\CCA^\subset=
\{
\CCB\subcl \exp X
\mid
\text{for all }A\in\CCA
\text{ there is }B\in\CCB
\text{ such that }B\subset A
\}
$$
is an inclusion hyperspace, i.e. an element of $G(\exp X)\subset
\exp^3X$.
\begin{lem}
The correspondence $\CCA\mapsto\CCA^\subset$ is a continuous
mapping $\exp^2X\to G(\exp X)$.
\end{lem}

\begin{proof}
Let $W$ be an open subset of $\exp X$, then the preimages of
subbase elements of $G(\exp X)$
\begin{gather*}
\{\CCA\in\exp^2X\mid
\CCA^\subset\in W^+
\}=
\{\CCA\in\exp^2X\mid
\text{for all }A\in\CCA\text{ there is }B\in W
\\
\text{ such that }B\subset A
\}=
\langle W\ups\rangle,
\end{gather*}
and
\begin{gather*}
\{\CCA\in\exp^2X\mid
\CCA^\subset\in W^-
\}=
\{\CCA\in\exp^2X\mid
\text{there is }A\in\CCA
\text{ such that }B\in W
\\
\text{ for all }B\in\exp X
\text{ such that }B\subset A\in\CCA
\}=
\langle\exp X,\exp X\setminus(\exp X\setminus W)\ups\rangle
\end{gather*}
are open, thus the mapping in question is continuous.
\end{proof}

\begin{cons}
For a closed relation $T\subset\exp^2X\times
\exp Y\times L$ the relation
$$
T^\subset=
\bigcup\{\CCA^\subset\times\{B\}\times \{\alpha\}\dns
\mid (\CCA,B,\alpha)\in T\}
$$
is closed, continuously depends on $T$, and
$(T^\subset)^\subset=T^\subset\supset T$.
\end{cons}

For a closed relation $T\subset\exp^2X\times
\exp Y\times L$ we define a relation $T^{\sup}\subset \exp^2 X\times
\exp Y\times L$ by the equality:
$$
T^{\sup}=
\{
(\bigcup\exp\pr_1(\CCF),\bigcup\exp\pr_2(\CCF),\sup\exp\pr_3(\CCF))
\mid
\CCF\subcl T, \CCF\ne\ES
\},
$$
where $\pr_i$, $i=1,2,3$, are the projections of the product $\exp
(\exp X)\times \exp Y\times L$ onto the respective factors.

\begin{lem}
The set $T^{\sup}$ is closed, satisfies
$(T^{\sup})^{\sup}=T^{\sup}\supset T$, and the mapping that takes
each $T$ to $T^{\sup}$ is continuous.
\end{lem}

The proof is obvious and uses the fact that for each compactum
$K$ the mapping $\bigcup:\exp^2K\to\exp K$ is continuous.\qed

\begin{proof}[Proof of the proposition]
For a strict $L$-ambiguous representation $R\subset \exp
X\times\exp Y\times L$ the set $R^\cup$ is equal to
$(R_\bullet)^{\sup}$, where
$$
R_\bullet= \{(\{A\},B,\alpha)\mid (A,B,\alpha)\in R\}.
$$
The closedness of $R_\bullet$ implies that $R^\cup$ is closed.

For each closed relation $T\subset\exp^2X\times
\exp Y\times L$ we put
$$
T^+=((T\cup \exp^2X\times \{Y\}\times L \cup
\exp^2X\times \exp Y\times \{0\})^\subset)^{\sup}.
$$
Then $T^+\subset\exp^2X\times
\exp Y\times L$ is closed, continuously depends on $T$, and
$(T^+)^+=T^+\supset T$. Moreover, $T=R^\cup$ for some
$R\in\Csamb_L(X,Y)$ if and only if
$$
T=T^+=(T\cap\{(\{A\},B,\alpha)\mid A\in\exp X,B\in\exp Y,\alpha\in
L\})^+.
$$
The latter equality selects a closed subset of $\exp(\exp^2X\times
\exp Y\times L)$.
\end{proof}

Therefore we define a compact Hausdorff topology on the set
$\Csamb_L(X,Y)$ by the requirement that the mapping that takes each
$R\in\Csamb_L(X,Y)$ to $R^\cup\in \exp(\exp^2X\times\exp Y\times
L)$ is an embedding.

For all $R,S\in \Csamb_L(X,Y)$ the inclusions $R\subset S$ and
$R^\cup\subset S^\cup$ are equivalent. The partial order on
$\Csamb_L(X,Y)$ is closed, hence for $\CCR\subset \Csamb_L(X,Y)$ we
have $\sup\CCR=\sup\Cl\CCR$. Therefore we further assume that
$\CCR$ is closed. Observe that $(\bigcup
\CCR)_\bullet=\bigcup\{R_\bullet\mid R\in\CCR\}$. For any upper
bound $S$ of $\CCR$ the relation $S^\cup$ should contain
$$
(\bigcup \CCR)^\cup =
(\bigcup\{R_\bullet\mid R\in\CCR\})^+=
(\bigcup\{(R_\bullet)^+\mid R\in\CCR\})^+=
(\bigcup\{R^\cup\mid R\in\CCR\})^+
$$
and the latter set satisfies the last equality from the proof of
the previous proposition. Therefore the least upper bound of $\CCR$
is determined by the equality
$$
(\sup\CCR)^\cup=
\bigl(\bigcup\{R^\cup\mid R\in\CCR\}\bigr)^+.
$$
This formula also implies that the mapping that takes each closed
set $\CCR$ to $\sup\CCR$ is continuous, hence $\Csamb_L(X,Y)$ is a
compact Lawson upper semilattice. The infimum in this lattice is in general not
continuous.

\begin{que}
When does $\Coamb_L(X,Y)$ have top and bottom elements? What are
topological properties of the subsets $\Cpsamb_L(X,Y)$,
$\Coamb_L(X,Y)$ of $\Csamb_L(X,Y)$?
\end{que}


\section{Epilogue}

Of course, interpretation of ambiguous representations is itself
somewhat ambiguous. We propose only one of possible (known to these
authors) possibilities. Let an \emph{object} be a closed subset $B$
of a compactum $Y$ (an \emph{object space}). The object is not
accessible by us directly, but we apply some procedures (series of
procedures) to elements of $B$ (or entire $B$) to obtain a closed
set $A$ in a compactum $X$ (a \emph{representation space}) that
represents (in some sense) the original set $B$. This information
is subject to random and systematic interferences, hence for a
fixed $B$ the result is \emph{ambiguous}, and even disjoint $A$ can
be obtained. This information can also be incomplete, therefore, if
a set $A$ represents a set $B$, then $A$ can also represent a
larger set $B'\supset B$. Likewise, if $A$ is obtained as a
representation of $B$, then any non-empty closed $A'\subset A$ can
also be obtained for the same $B$, e.g. if less attempts to obtain
information have been made.

It is natural to demand that the relation "$A$ can represent $B$"
be closed (=topologically stable), i.e. if $A$ are valid
representations of $B_i$, and $B_i$ converge to $B_0$, then $A$
should also be a valid representation of $B_0$. If the same
(optionally) is true for representing sets, then the representation
is called
\emph{strict}. Thus we obtain a binary relation $R$ between the sets
$\exp X$ and $\exp Y$ of non-empty closed subsets of $X$ and $Y$,
and strict representations are characterized by the property that
they are closed in the product of $\exp X$ and $\exp Y$ with the
Vietoris topologies.

A "pseudo-inverse" to $R$ ambiguous representation $R^\sms\subset
\exp Y\times \exp X$ appears when, given a representing set in $X$,
we are interested in areas in $Y$ which are \emph{outside} of the
object. Namely, $(\tilde B,\tilde A)\in R^\sms$ if and only if all
closed non-empty $A$ outside of $\tilde A$ can represent via $R$ some
non-empty closed $B$ outside of $\tilde B$. If $(R^\sms)^\sms=R$,
we call $R$ \emph{pseudo-invertible}. The relation $R^\sms$
is closed only if $R$ satisfies a requirement similar to openness
of a mapping. If the equality $(R^\sms)^\sms=R$ is also valid, we
call such $R$ an~\emph{open ambiguous representation}, and $R^\sms$
belongs to the same class.

For ambiguous representations the composition law is defined,
which expresses formally an intuitive fact that representations
(procedures of obtaining information) can be combined sequentially.
Although this law is not associative in general, strict,
pseudo-invertible and open representations form respectively
categories $\Csamb$, $\Cpamb$ and $\Coamb$. The operation $(-)^\sms$
determines antiisomorphisms from the categories $\Cpamb$ and $\Coamb$
onto themselves.

It can be useful that, for given compacta $X$, $Y$, the set
$\Csamb(X,Y)$ of strict ambiguous representations between $X$ and
$Y$ is a compact Hausdorff space and a complete lattice with
respect to inclusion, and open representations form a sublattice
$\Coamb(X,Y)$. It allows one to compare and approximate
representations.

It is natural that a set $A\subset X$ can represent different
$B\subset Y$ with different level of acceptability. We propose to
express this level as an element of a lattice $L$. An
$L$-fuzzification of the above theory is also provided in the
paper. To interact well with compacta, $L$ must be a compact
Hausdorff Lawson lattice (possess local bases that consist of
sublattices). This class includes the most common case $L=[0,1]$.
Then a
\emph{strict $L$-ambiguous representation} is a closed $L$-relation
between $\exp X$ and $\exp Y$, i.e. a closed subset $R\subset \exp
X\times\exp Y\times L$, with certain properties. An equivalent, but
sometimes more convenient interpretation: for all $A\in \exp X$ we
fix an~$L$\emph{-capacity} $c_{AR}$ (cf.~\cite{Nyk:CapLat:08}). It
is a function that sends each $B\in\exp Y$ to an element
$c_{AR}(B)\in L$ that shows how appropriate is $A$ as a
representation of $B$ (the more, the better). To reflect the fact
that successive application of uncertain conclusions can give even
more uncertain result, we propose to use generalized triangular
norms $*$ on $L$~\cite{Dross:GenTNormStr:99}
to define compositions. Thus we obtain a collection of
categories $\Csamb^*_L$, $\Cpamb^*_L$, $\Cpsamb^*_L$, and $\Coamb^*_L$.

The~definition of fuzzy ambiguous representation allows at least as many
interpretations as the definitions of fuzzy set and fuzzy relation. Hence
we shall not discuss them here and refer the reader to
\cite{DubPr:ThreeSemFuzzSets:97,DubPr:GradUncBip-MakSense:10}.
Virtually any of semantics of fuzzy sets considered in the latter citations
can be meaningfully applied to the objects defined in this paper.

We expect that ambiguous representations will become a convenient
framework for problems of image recognition and data mining,
allowing to apply methods of topology and category theory.
Interplay between this theory, fuzzy and rough sets/relations, and
functors and monads in the category of compacta will be the topic
of our next paper.

The authors wish to express gratitude to anonymous reviewers and
editors for valuable comments, corrections and suggestions.

\end{document}